\documentclass{article}

\usepackage{amsmath, amsthm, amsfonts, amssymb, enumerate, esint, dsfont}
\usepackage{mathtools}
\usepackage[margin={1.6in,1.4in}]{geometry}

\usepackage[colorlinks=true,citecolor=blue]{hyperref}

\usepackage[dvipsnames]{xcolor}

\usepackage{todonotes}
\newtheorem{theorem}{Theorem}
\newtheorem{lemma}{Lemma}
\newtheorem{proposition}{Proposition}

\theoremstyle{definition}

\theoremstyle{remark}
\newtheorem{remark}{Remark}

\newcommand{\ud}{\,\mathrm{d}}
\newcommand{\R}{\mathbb{R}}
\newcommand{\Rd}{\mathbb{R}^d}

\newcommand{\Zd}{\mathbb{Z}^d}

\newcommand{\Td}{\mathbb{T}^d}

\newcommand{\ind}{\mathrm{1}}

\newcommand{\law}{\mathcal{L}}
\newcommand{\wconv}{\to^{\law}}
\newcommand{\G}{\mathbb{G}}
\newcommand{\F}{\mathcal{F}}

\begin{document}
	\title{Donsker Theorems for Occupation Measures of Multi-Dimensional Periodic Diffusions}
	
	\author{
		Neil Deo\thanks{University of Cambridge (\texttt{and30@cam.ac.uk})}
	}
	\date{}
	\maketitle
	
	\begin{abstract}
		We study the empirical process arising from a multi-dimensional diffusion process with periodic drift and diffusivity. The smoothing properties of the generator of the diffusion are exploited to prove the Donsker property for certain classes of smooth functions. We partially generalise the finding from the one-dimensional case studied in \cite{vandervaartDonskerTheoremsDiffusions2005}: that the diffusion empirical process exhibits stronger regularity than in the classical case of i.i.d. observations. As an application, precise asymptotics are deduced for the Wasserstein-1 distance between the time-$T$ occupation measure and the invariant measure in dimensions $d\leq3$.
	\end{abstract}
	
	\section{Introduction}
	
	Consider the multi-dimensional diffusion process $(X_t)_{t\ge0}$ which is the solution to the stochastic differential equation (SDE)
	\begin{equation}\label{Eq: Diffusion SDE}
		\ud X_t = b(X_t)\ud t + \sigma(X_t)\ud W_t, \quad X_0 = x_0\in\Rd, \quad t\geq0.
	\end{equation}
	Here, $(W_t)_{t\geq0}$ is a $d$-dimensional Brownian motion, and the \emph{drift field} $b:\Rd\to\Rd$ and the \emph{diffusivity} $\sigma:\Rd\to\R^{d\times d}$ are assumed to be Lipschitz continuous; thus the law of $X$ is uniquely defined on the cylindrical sigma-algebra of $[0,\infty)\times\Rd$. Under mild conditions on $b$ and $\sigma$, as $t\to\infty$ the law of $X_t$ converges to the \emph{invariant measure} $\mu = \mu_{b,\sigma}$, defined on a suitable state space. Moreover, the \emph{ergodic theorem} holds for suitable measurable sets $A\subset\Rd$:
	\begin{equation}\label{Eq: ergodic property of mu, sets}
		\hat{\mu}_T(A) := \frac{1}{T}\int_0^T \ind_A(X_t)\,\ud t \underset{T\to\infty}{\to} \mu(A)\quad\text{almost surely.}
	\end{equation}
	This random measure $\hat{\mu}_T$, called the \emph{time-$T$ occupation measure}, is the natural analogue of the empirical measure in the diffusion setting. The ergodic property does not only hold for sets, but also for sufficiently regular functions $f$:
	\begin{equation}\label{Eq: ergodic property of mu, functions}
		\hat{\mu}_T(f) := \frac{1}{T}\int_0^T f(X_t)\,\ud t \underset{T\to\infty}{\to} \int f\,\ud\mu =:\mu(f) \quad\text{almost surely.}
	\end{equation}
	
	With the ergodic property (\ref{Eq: ergodic property of mu, functions}) in hand, the next natural question is whether or not $\hat{\mu}_T(f)$ obeys a central limit theorem, and indeed this is the case (see, for example, \cite{bhattacharyaFunctionalCentralLimit1982}): for suitable functions $f$, we have that
	\begin{equation}\label{Eq: one function CLT}
		\sqrt{T}\left(\hat{\mu}_T(f) - \mu(f)\right) \overset{\law}{\to} N(0,\Gamma(f,f)) \quad \text{ as }T\to\infty,
	\end{equation}
	where the covariance operator $\Gamma$ is determined by $b$ and $\sigma$. By the Cram\'{e}r-Wold device, an analogous result holds for any suitable finite-dimensional class of functions $\F$. 
	
	In this paper, we investigate whether such a central limit theorem holds uniformly over a potentially \emph{infinite-dimensional} class $\F$. This question was addressed in the scalar diffusion setting ($d=1$) in \cite{vandervaartDonskerTheoremsDiffusions2005}: they prove that such a central limit theorem holds whenever the limit process exists as a tight Borel probability measure on $\ell^{\infty}(\F)$, the set of bounded functions on $\F$. This contrasts with the i.i.d. sampling model, in which the existence of the limiting process as a tight Borel measure is not sufficient for a uniform central limit theorem to hold: see Proposition 4.4.7 and the preceding discussion in \cite{gineMathematicalFoundationsInfinitedimensional2016}. In the scalar diffusion model, no such gap exists. Moreover, as (\ref{Eq: one function CLT}) dictates the covariance structure of the limit, one deduces necessary and sufficient conditions based on the class $\F$ and the operator $\Gamma$ for a Donsker theorem to hold over $\F$. Using these conditions, it is shown that the diffusion empirical process is smoother than its i.i.d. counterpart, in the sense that there exist function classes $\F$ which are Donsker for the diffusion empirical process but not for a classical (i.i.d.) empirical process. 
	
	However, the case of scalar diffusions is considerably easier than the multidimensional case due to the existence of local time, on  which the proofs of \cite{vandervaartDonskerTheoremsDiffusions2005} and several other works rely. Diffusion local time ceases to exist when $d\geq2$, and so alternative methods must be used to deduce concentration properties of the diffusion process $X$.
	Here, we use the approach of \cite{nicklNonparametricStatisticalInference2020}, based on an analysis of the \emph{generator} of the diffusion process as a partial differential operator. While other works on the statistics of multi-dimensional diffusions have relied on functional inequalities to deduce the necessary concentration properties of $(X_t)$ \cite{dalalyanAsymptoticStatisticalEquivalence2007,strauchAdaptiveInvariantDensity2018}, this requires the diffusion to be \emph{reversible}, or equivalently requires that $b$ is a gradient vector field. The PDE arguments of \cite{nicklNonparametricStatisticalInference2020} do not require this additional structural assumption. We restrict to the \emph{periodic case} (as in \cite{giordanoNonparametricBayesianInference2022,nicklNonparametricStatisticalInference2020}), in which $b = b(\cdot+k)$ and $\sigma = \sigma(\cdot+k)$ for any $k\in\Zd$. This effectively confines the diffusion process $X$ to the $d$-dimensional torus, in the sense that the values of $X\mod\Zd$ contain all of the statistical information about $b$ and $\sigma$. Moreover, although the process $X$ is not recurrent on $\Rd$, the invariant measure $\mu$ is well-defined on the torus, and it is in this sense that the ergodic and central limit theorems (\ref{Eq: ergodic property of mu, sets})-(\ref{Eq: one function CLT}) hold. For a discussion on the non-periodic case, see Section \ref{Subsection: Discussion}.
	
	Using these techniques, we uncover a similar story to the one-dimensional case: that the empirical process for diffusions exhibits stronger regularity than in the classical i.i.d. setting. As a demonstration, Theorem \ref{Thm: Wasserstein CLT} below implies that the weak convergence of $\hat{\mu}_T$ to $\mu$ in Wasserstein space occurs at a rate of $T^{-1/2}$ in dimensions $d\leq3$; in the i.i.d. setting, the parametric rate of convergence only occurs when $d=1$. This increased regularity is due to the smoothing properties of the inverse of the generator of the diffusion process, which appears in the covariance structure of the limit process: see (\ref{Eq: limit process cov kernel}) below.

	\section{Main Results}
	
	\subsection{Definitions \& Notation}
	
	Let $\Td$ be the $d$-dimensional torus, isomorphic to $(0,1]^d$ if opposite points on the cube are identified. We denote by $L^2(\Td)$ the usual $L^2$ space with respect to the Lebesgue measure $\ud x$, equipped with its inner product $\langle f, g \rangle = \langle f,g \rangle_{L^2} := \int fg \,\ud x$. (If ever a region of integration is unspecified, it should be taken to be the whole of $\Td$.) Let $\mu$ be a Borel probability measure on $\Td$ which is absolutely continuous with respect to Lebesgue measure, and whose density is bounded and bounded away from zero. Then the inner product $\langle f,g \rangle_{\mu} := \int fg\,\ud\mu$ induces an equivalent norm on $L^2(\Td)$. We further define the subspaces
	$$ L^2_0(\Td):= \left\{ f\in L^2(\Td): \int f\,\ud x=0 \right\}, \quad L^2_{\mu}(\Td) := \left\{ f\in L^2(\Td): \int f\,\ud \mu =0 \right\}. $$
	
	For $s\in\R$, let $H^s(\Td)$ denote the usual Sobolev space of functions, defined by duality when $s<0$. Recall that $L^2(\Td) = H^0(\Td)$. The Sobolev spaces comprise the case $p=q=2$ on the scale of Besov spaces $B^s_{pq}(\Td)$, \mbox{$s\in\R,1\leq p,q\leq \infty$}; see Chapter 4 of \cite{gineMathematicalFoundationsInfinitedimensional2016} for definitions. The usual (isotropic) H\"{o}lder spaces $C^s(\Td)$, $s\geq0$ are closely related to the spaces $B^s_{\infty\infty},s\geq0$: specifically, \mbox{$C^s(\Td) = B^s_{\infty\infty}(\Td)$} with equivalent norms for non-integer $s\geq0$, and  \mbox{$C^s(\Td) \subsetneq B^s_{\infty\infty}(\Td)$} when $s\in\mathbb{N}\cup\{0\}$. Finally, we introduce the $L^p$-Sobolev spaces: for $k\in\mathbb{N}$ and $1\leq p\leq\infty$, we define
	$$ W^{k,p}(\Td) := \left\{ f\in L^p(\Td): D^kf\text{ exists and }D^kf\in L^p\left((\Td)^{\otimes k}\right) \right\}; $$
	here $D^kf$ is the $k^{th}$-order weak derivative of $f$. Note that for $k\in\mathbb{N}$, \mbox{$W^{k,2}(\Td) = H^k(\Td)$}.
	
	We employ the same notation for vector fields \mbox{$f = (f_1,\ldots,f_d)$}, by a slight abuse of notation: in this case, $f\in H^s(\Td) = (H^s(\Td))^{\otimes d}$ means that each $f_j\in H^s(\Td)$, with norm $\|f\|_{H^s}^2 = \sum_{j=1}^d \|f_j\|_{H^s}^2$. When no confusion may arise, we omit the domain and simply write $C^s, H^s,$ etc. The codomain will always be clear from context.
	
	Let $\mu$ be a (Borel) probability measure on some metric space. We write $Z\sim\mu$ if $Z$ is a random variable with law $\mu$. Given a collection of random variables $(Z_T)_{T\geq0}$, we then write $Z_T \to^{\law} Z$, or $Z_T \wconv \mu$, to denote the usual notion of weak convergence as $T\to\infty$, see Chapter 11 of \cite{dudleyRealAnalysisProbability2002} or Section 3.7 of \cite{gineMathematicalFoundationsInfinitedimensional2016} for definitions. We also write $Z_T\to^P Z$ to mean that $Z_T$ converges in probability to $Z$ as $T\to\infty$ under the measure $P$.
	
	Let $(Z_f)_{f\in\F}$ be a stochastic process indexed by a set $\F$. Assume that the process $Z$ is bounded, and its finite-dimensional marginals correspond to those of a tight Borel probability measure on $\ell^{\infty}(\F)$, the space of all bounded functions $\F\to\R$, equipped with the supremum norm $\|\cdot\|_{\F}$. Then $Z$ has a measurable version taking values in a separable subset of $\ell^{\infty}(\F)$ almost surely, which we take to be $Z$ itself. Given bounded processes $(Z^n_f)_{f\in \F}$, $n\geq 1$, we say that \emph{$Z^n$ converges in law to $Z$ in $\ell^{\infty}(\F)$}, written
	\begin{equation}\label{Eq: weak convergence definition}
		Z^n \wconv Z \quad \text{in }\ell^{\infty}(\F),
	\end{equation}
	if for every bounded, continuous function $H:\ell^{\infty}(\F)\to\R$, we have
	\mbox{$ E^*H(Z^n) \to E H(Z)$},
	where $E^*$ denotes outer expectation.
	
	We will sometimes use the symbols $\lesssim,\gtrsim,\simeq$ to denote one- or two-sided inequalities which hold up to multiplicative constants that may be either universal or `fixed' in the relevant context (in the latter case, this will always be qualified).
	Finally, we write $\partial_i$ as shorthand for the weak partial derivative $\frac{\partial}{\partial x_i}$ on $\Rd$ or $\Td$, for $1\leq i \leq d$, and $\nabla$ for the weak gradient.

	\subsection{Diffusions with Periodic Drift; Invariant Measures}
	
	We consider the SDE (\ref{Eq: Diffusion SDE}) where the drift \mbox{$b:\Rd\to\Rd$} and diffusivity \mbox{$\sigma:\Rd\to\R^{d\times d}$} are Lipschitz continuous and one-periodic, so that \mbox{$b(x) = b(x+k)$}, \mbox{$\sigma(x)=\sigma(x+k)$} for all $x\in\Rd,k\in\Zd$. We may therefore consider $b,\sigma$ as maps on $\Td$; in general, we identify functions on the torus with their periodic extensions to $\Rd$ in this way. Under this assumption, a strong pathwise solution $(X_t)_{t\geq0}$ of the SDE (\ref{Eq: Diffusion SDE}) exists (e.g. Theorem 24.3 in \cite{bassStochasticProcessesRichard2011}); we write $P=P_{b,\sigma}=P_{b,\sigma,x}$ for the law of $X$ when $X_0=x$ on the cylindrical sigma-algebra of the path space $C([0,\infty)\to\Rd)$. We suppress the dependence on the initial value $x$ as our results do not depend on it; we further suppress dependence on $b$ and $\sigma$ where no confusion may arise, as we consider these to be fixed.
	
	As discussed previously, the periodicity of the drift $b$ and diffusivity $\sigma$ permit us to consider the diffusion $X$ to take values in $\Td$, without the loss of any statistical information about the parameters $b,\sigma$. (Note that the periodicity of $b$ means that $X$ is not even recurrent as a process on $\Rd$.) Once viewed as taking values in $\Td$, the process $X$ possesses an invariant measure $\mu = \mu_{b,\sigma}$ under mild conditions on $b$ and $\sigma$ (see Proposition \ref{Prop: inv measure existence} below), and the ergodic property \eqref{Eq: ergodic property of mu, functions} holds for all $f\in C(\Td)$.
	
	However, rather than viewing $\mu_{b,\sigma}$ as an object determining ergodic averages, we consider it as the solution of a certain partial differential equation arising from the \emph{infinitesimal generator} of the diffusion process $X$, defined as the second order partial differential operator \mbox{$L:H^2(\Td)\to L^2(\Td)$} given by
	\begin{equation}\label{Eq: generator definition}
		L = L_{b,\sigma} := \sum_{i,j=1}^d a_{ij}(\cdot)\partial_i\partial_j+ \sum_{i=1}^d b_i(\cdot)\partial_i
		= \sum_{i=1}^d \partial_i\left(\sum_{j=1}^d a_{ij}(\cdot)\partial_j \right) + \sum_{i=1}^d \tilde{b}_i(\cdot)\partial_i,
	\end{equation}
	where $a := \frac{1}{2}\sigma\sigma^{\top}$ and $\tilde{b}:\Td\to\Rd$ is defined by $\tilde{b}_i = b_i - \sum_j\partial_ja_{ij}$.
	Integrating by parts, one computes the $L^2$-adjoint operator of $L$ to be
	\begin{equation}\label{Eq: adjoint generator definition}
		L^* = L^*_{b,\sigma} = \sum_{i=1}^d \partial_i\left(\sum_{j=1}^d a_{ij}(\cdot)\partial_j \right) - \sum_{i=1}^d \tilde{b}_i(\cdot)\partial_i - \sum_{i=1}^d\partial_i \tilde{b}_i.
	\end{equation}
	A standard argument (as in Section 2.2 of \cite{nicklNonparametricStatisticalInference2020}, for example) establishes that the invariant measure $\mu$ must then weakly solve
	\begin{equation}\label{Eq: adjoint pde for inv measure}
		L^*\mu = 0 \quad \text{in }\Td.
	\end{equation}
	As $\sigma:\Td\to\R^{d\times d}$ is assumed to be Lipschitz, it is uniformly bounded. We further assume that $\sigma$ is such that the induced $a = \frac{1}{2}\sigma\sigma^{\top}$ is \emph{strictly elliptic}: that is, there exist constants $0<\lambda,\Lambda<\infty$ such that
	\begin{equation}\label{Eq: strict ellipticity definition}
		\lambda |\xi|^2 \leq \sum_{i,j=1}^d a_{ij}(x)\xi_i\xi_j \leq \Lambda |\xi|^2, \quad \forall x\in\Td,\xi\in\Rd.
	\end{equation}
	Under this assumption, the second-order differential operators $L$ and $L^*$ are strictly elliptic and one may utilise the standard theory for such operators (\cite{bersPartialDifferentialEquations1964,gilbargEllipticPartialDifferential2001}). 
	
	For sufficiently regular $b,\sigma$, the adjoint equation (\ref{Eq: adjoint pde for inv measure}) characterises the invariant measure, as shown by the following result.
	\begin{proposition} \label{Prop: inv measure existence}
		Assume $b\in C^1(\Td),\sigma\in C^2(\Td)$, and assume further that $a=\frac{1}{2}\sigma\sigma^{\top}$ satisfies (\ref{Eq: strict ellipticity definition}). Then a unique periodic solution $\mu=\mu_{b,\sigma}$ to (\ref{Eq: adjoint pde for inv measure}) satisfying $\int \ud\mu =1$ exists. Moreover, $\mu$ is Lipschitz and bounded away from zero on $\Td$. 
	\end{proposition}
	The proof is identical to that of \cite[Proposition 1]{nicklNonparametricStatisticalInference2020}, which covers the special case $\sigma\equiv \mathrm{id}$. The proof extends naturally to the strictly elliptic case considered here using standard elliptic PDE theory; in particular, the relevant heat kernel estimates from \cite{norrisLongTimeBehaviourHeat1997} hold for such strictly elliptic operators.
	Furthermore, standard arguments tell us that there is a well-defined \emph{solution operator} $L^{-1}:L^2(\Td)\to H^2(\Td)$ such that
	\begin{equation}\label{Eq: solution operator definition}
		L(L^{-1}[f]) = f \quad \forall f\in L^2_{\mu}.
	\end{equation}
	Moreover, $L^{-1}$ is a bounded linear operator from $L^2$ to $H^2$, and more generally can be thought of as a `2-smoothing' operator between Besov spaces: see Lemma \ref{Lemma: 2-smoothing estimate} below. 
	
	\subsection{Central Limit Theorems for $\hat{\mu}_T$}
	
	Our central limit theorem concerns the empirical process
	\begin{equation}\label{Eq: empirical process definition}
		\mathbb{G}_T(f) := \sqrt{T}\left( \hat{\mu}_T(f) - \mu(f)\right), \quad f\in\F,
	\end{equation}
	where $\F$ is some class of test functions. We wish to establish the weak convergence of $\G_T$ to some limit process $\G$ in $\ell^{\infty}(\F)$, in the sense of (\ref{Eq: weak convergence definition}). This is accomplished by checking the following two conditions (see \cite[Theorem 3.7.23]{gineMathematicalFoundationsInfinitedimensional2016} or \cite[Theorems 1.5.4 and 1.5.7]{vaartWeakConvergenceEmpirical1996}):
	\begin{enumerate}
		\item The finite-dimensional marginals of $\G_T$ converge weakly to those of $\G$.
		\item There exists a pseudo-metric $\rho$ on $\F$ such that $(\F,\rho)$ is totally bounded and for all $\varepsilon>0$,
		\begin{equation}\label{Eq: asymptotic equicontinuity condition}
			\lim_{\delta\to0}\limsup_{T\to\infty} \Pr\left\{ \sup_{\rho(f,g)\leq\delta} 		\left|\G_T(f)-\G_T(g)\right|>\varepsilon \right\} = 0.
		\end{equation}
	\end{enumerate}
	We refer to (\ref{Eq: asymptotic equicontinuity condition}) as the \emph{asymptotic equicontinuity condition}; it ensures that the sequence $\G_T$ is asymptotically tight, and consequently that the limit process $\G$ exists as a tight Borel probability measure on $\ell^{\infty}(\F)$.
	
	The first condition determines the form of the limiting process $\G$ (when it exists): an application of the martingale central limit theorem (see e.g. \cite[Theorem 7.1.4]{ethierMarkovProcessesCharacterization1986}) dictates the covariance structure of $\G$ to be given by \eqref{Eq: limit process cov kernel} below.
	It then remains to check that the asymptotic equicontinuity condition holds. Via chaining arguments, this can be reduced to suitably controlling the complexity of the class $\F$ as measured by metric entropy with respect to the pseudo-distance $\rho$, for a `good' choice of $\rho$. Here, one may leverage the smoothing property of the solution operator $L^{-1}$ to define a rather weak pseudo-distance $\rho$ (see (\ref{Eq: d_L definition}) below), with respect to which the process $\G_T$ is sub-Gaussian uniformly in $T$. These considerations lead to the following result.
	
	\begin{theorem}\label{Thm: Main CLT}
		Let $d\geq2$. Assume that there exists $\beta>0$ such that $b\in C^{1+\beta}(\Td)$, and that $\sigma\in C^2(\Td)$ satisfies the strict ellipticity condition (\ref{Eq: strict ellipticity definition}) for some $\lambda>0$. Let $1\leq p,q\leq\infty$, $s>\left(\max\left\{\frac{d}{2},\frac{d}{p}\right\}-1\right)$, and let $\F$ be a bounded subset of $B^s_{pq}$. Let $\G_T$ be the empirical process defined in (\ref{Eq: empirical process definition}). Then as $T\to\infty$,
		$$ \G_T \to^{\law} \G \quad \text{in }\ell^{\infty}(\F),$$
		where $\G = \G_{b,\sigma}$ is centred Gaussian process on $\F$ with covariance kernel
		\begin{equation}\label{Eq: limit process cov kernel}
			E \G(f)\G(g) = \langle\sigma^{\top} \nabla L^{-1}[\overline{f}],\sigma^{\top}\nabla L^{-1}[\overline{g}] \rangle_{\mu},
		\end{equation}
		whose law defines a tight Borel probability measure on $\ell^{\infty}(\F)$. Here, $\mu=\mu_{b,\sigma}$ is the invariant measure of the diffusion process with drift field $b$ and diffusivity $\sigma$, $L^{-1}=L^{-1}_{b,\sigma}$ is the inverse operator of the generator of the diffusion as defined in \eqref{Eq: solution operator definition}, and for $f\in\F$, $\overline{f}:=f-\int f\ud\mu$.
	\end{theorem}
	\begin{remark}
		The theorem also holds (with an identical proof) in the case $d=1$ whenever $s>0$. Note that the empirical process is more regular than in the i.i.d. case, which requires $s>d/2$ in any dimension for the limit to even be tight on $\ell^{\infty}(\F)$.
	\end{remark}
	The next proposition establishes that the assumption $s>\frac{d}{2}-1$ is necessary and sufficient for the limit process $\G$ to exist as a tight Borel measure on $\ell^{\infty}(\F)$. 		
	\begin{proposition}\label{Prop: sharpness of CLT}
		Let $d\geq2$, and let $b,\sigma,\mu=\mu_{b,\sigma},L^{-1}=L^{-1}_{b,\sigma}$ be as in Theorem \ref{Thm: Main CLT}. Let $1\leq p,q\leq \infty$, $s\geq0$, and let $\F$ be the unit ball of $B^s_{pq}(\Td)$. Then the centred Gaussian process $\G$ defined by (\ref{Eq: limit process cov kernel}) is a tight Borel measure on $\ell^{\infty}(\F)$ if and only if $s>\frac{d}{2}-1$.
	\end{proposition}
	Thus in the case $p\in[2,\infty]$, the condition $s>\frac{d}{2}-1$ in Theorem \ref{Thm: Main CLT} is sharp. However, in the case $1\leq p<2$, when \mbox{$s\in(\frac{d}{2}-1,\frac{d}{p}-1]$}, while the limit process $\G$ exists as a tight Borel measure on $\ell^{\infty}(\F)$, we do not know whether the empirical process $\G_T$ converges to it. For the analogous parameter configuration in the i.i.d. setting, $s\in\left(d/2,d/p\right]$, the limit exists as a tight Borel measure, but the empirical process does not in fact converge to this limit (\cite[Proposition 4.4.7]{gineMathematicalFoundationsInfinitedimensional2016}). In the scalar diffusion setting, it was established in \cite{vandervaartDonskerTheoremsDiffusions2005} that no such phenomenon takes place, and the empirical process converges to the limit experiment whenever it exists. What occurs in the multidimensional case is currently unclear; see Section \ref{Subsection: Discussion} for a discussion of this matter. 
	
	Theorem \ref{Thm: Main CLT} is related to the Bernstein-von-Mises results in \cite{nicklNonparametricStatisticalInference2020} (Theorems 5 and 6) in the case $\sigma \equiv\mathrm{id}$ and $d\leq3$. In place of the occupation measure $\hat{\mu}_T$, their results are for a sequence $\mu_{\hat{b}_T}$, where $\hat{b}_T$ is a suitable estimator of the drift field $b$ (the posterior mean), and $\mu_{\hat{b}_T}$ denotes the invariant measure associated to the drift field $\hat{b}_T$. In the frequentist literature, a uniform central limit theorem similar to Theorem \ref{Thm: Main CLT} was deduced in the scalar setting for a suitable kernel estimator $\hat{b}_T$ in \cite{aeckerle-willemsSupnormAdaptiveDrift2022}, building on the results of \cite{vandervaartDonskerTheoremsDiffusions2005}.
	
	As an application of Theorem \ref{Thm: Main CLT}, we derive convergence rates for the occupation measure $\hat{\mu}_T$ in the Wasserstein distance in dimensions $d\leq3$. Recall that the Wasserstein-1 distance between Borel measures $\nu,\rho$ on $\Td$ is defined as
	\begin{equation}\label{Eq: W_1 definition}
		W_1(\nu,\rho) = \inf_{\gamma\in\mathcal{C}(\nu,\rho)} \int_{\Td\times\Td} \|x-y\|\,\ud\gamma(x,y),
	\end{equation}
	where $\|\cdot\|$ denotes the Euclidean norm on $\Td$ and $\mathcal{C}(\nu,\rho)$ is the set of all \emph{couplings} of $\nu$ and $\rho$: Borel measures on $\Td\times\Td$ whose first and second marginals (on $\Td$) are $\nu$ and $\rho$ respectively. By the Kantorovich-Rubinstein duality formula, one may rewrite the Wasserstein-1 distance as
	\begin{equation}\label{Eq: K-R duality}
		W_1(\hat{\mu}_T, \mu) = \sup_{f\in\mathrm{Lip}_1(\Td)} \left[\int f\ud\hat{\mu}_T - \int f\ud\mu \right],
	\end{equation}
	where $\mathrm{Lip}_1(\Td)$ is the class of all functions $f:\Td\to\R$ that are 1-Lipschitz.
	
	\begin{theorem}\label{Thm: Wasserstein CLT}
		Let $d\leq3$. Let $b,\sigma$ be as in Theorem \ref{Thm: Main CLT}; let $\mu$ be the invariant measure of the associated diffusion process and let $\hat{\mu}_T$ be the time-$T$ occupation measure as defined in (\ref{Eq: ergodic property of mu, sets}). Then
		$$ \sqrt{T}\,W_1(\hat{\mu}_T, \mu) \to^{\law} \|\G\|_{\F} \quad \text{as }T\to\infty,$$
		where $\G$ is the centred Gaussian process on $\F:=\mathrm{Lip}_1(\Td)\cap L^2_{\mu}(\Td)$ with covariance kernel given by (\ref{Eq: limit process cov kernel}), and $\|\G\|_{\F}$ is finite almost surely.
		\begin{proof}
			By linearity, one may replace $\mathrm{Lip}_1(\Td)$ in (\ref{Eq: K-R duality}) with the class $\F$ defined in the statement; observe that $\F$ is a bounded subset of $B^1_{\infty\infty}(\Td)$. Theorem \ref{Thm: Main CLT} now applies with $s=1$ and $d\leq3$, and this gives the result after noting that $W_1(\hat{\mu}_T,\mu)$ is measurable (as the supremum in (\ref{Eq: K-R duality}) is attained over a countable set) and $\|\cdot\|_{\F}:\ell^{\infty}(\F)\to\R$ is continuous.
			
			Lastly, by considering the covariance structure of $\G$, one shows that $\|\G\|_{\F}$ is sub-Gaussian with respect to a multiple of the distance $\rho_L$ defined in (\ref{Eq: d_L definition}) below. Analogous arguments to those of Section \ref{Subsection: Proof, Asymptotic Equicontinuity} establish that $\|\G\|_{\F}$ is almost surely finite.
		\end{proof}
	\end{theorem} 
	\begin{remark}(Concentration, finite sample bounds)
		Using Exercise 2.3.1 of \cite{gineMathematicalFoundationsInfinitedimensional2016}, one can show that $\|\G\|_{\F}$ concentrates about its finite expectation with sub-Gaussian tails. As was noted in Lemma 1 of \cite{nicklNonparametricStatisticalInference2020}, a similar analysis of $\|\G_T\|_{\F}$ can be used to derive non-asymptotic sub-Gaussian concentration bounds for $W_1(\hat{\mu}_T,\mu)$. Similar finite sample results were proved in \cite{boissardSimpleBoundsConvergence2011}, using transport-entropy inequalities.
	\end{remark}
	
	In particular, Theorem \ref{Thm: Wasserstein CLT} tells us that in dimensions $d\leq 3$,
	$$W_1(\hat{\mu}_T,\mu) = O_P(T^{-1/2}).$$ 
	When $d=2,3$, the empirical measure in the diffusion model converges faster than for a smooth density in the i.i.d. sampling model, in which the empirical measure converges in $W_1$-distance at a rate of $n^{-1/2}(\log{n})^{1/2}$ for $d=2$ and $n^{-1/3}$ for $d=3$, where $n$ is the sample size (\cite{ajtaiOptimalMatchings1984,dudleySpeedMeanGlivenkoCantelli1968}; see also \cite{fournierRateConvergenceWasserstein2015}). The $d=1$ case of the classical i.i.d. setting, in which $W_1$ convergence does occur at the parametric rate, was comprehensively analysed in \cite{delbarrioCentralLimitTheorems1999}. 
	
	\subsection{Discussion} \label{Subsection: Discussion}
	
	A shortcoming of Theorem \ref{Thm: Main CLT} is that it fails to address the case $1\leq p<2$, \mbox{$s\in(\frac{d}{2}-1,\frac{d}{p}-1]$}. For these parameters, the limit process $\G$ defined by \eqref{Eq: limit process cov kernel} exists as a tight Borel measure on $\ell^{\infty}(\F)$ by Proposition \ref{Prop: sharpness of CLT}. However, the relevance of the limit to the empirical process $\G_T$ is no longer clear, since the It\^{o}-Krylov formula \eqref{Eq: Ito decomposition of G_T} does not necessarily hold. This formula from \cite{krylovControlledDiffusionProcesses1980} appears to be the most general such result for weakly differentiable functions, holding for functions in $W^{2,r}(\Td)$ for any $r>d$. In fact, stochastic chain rules exist for only once weakly differentiable functions, as in \cite{altmeyerFourierMethodsEstimating2017,follmerItoFormulaMultidimensional2000,russoItoFormulaClfunctions1996}; however, in these formulae the quadratic covariation term is replaced by a suitable analogue, and so they do not yield an occupation time formula as in \eqref{Eq: Ito decomposition of G_T} below. Moreover, such formulae require delicate initial conditions, often depending on the function $f$ under consideration, and so are unsuitable for developing uniform central limit theorems.
	
	One may ask whether the assumption of periodicity of the drift field $b$ and the diffusivity $\sigma$ may be removed. In principle, this should be possible by making use of the results of \cite{pardouxPoissonEquationDiffusion2001} concerning the Poisson equation on $\Rd$, as in \cite{aeckerle-willemsSupnormAdaptiveDrift2022}. However, it is not obvious how precisely to adapt our key regularity estimate Lemma \ref{Lemma: 2-smoothing estimate} below; in particular, the choice of which Besov spaces to use must be informed by metric entropy considerations if there is to be any hope of proving the asymptotic equicontinuity condition (\ref{Eq: asymptotic equicontinuity, expectation version}) via chaining methods. One candidate is the family of weighted Besov spaces studied in \cite{haroskeWaveletBasesEntropy2005} and successfully deployed in \cite{nicklBracketingMetricEntropy2007} in the classical i.i.d. empirical process setting. However, this generalisation requires many complicated technicalities to be resolved, and is left as a matter for future research.
	
	\section{Proof of Theorem \ref{Thm: Main CLT}}
	
	For this section, we fix the dimension $d\geq2$, as well as a drift field $b$ and diffusivity $\sigma$ satisfying the conditions of Theorem \ref{Thm: Main CLT}. Fix $1\leq p,q\leq\infty$, \mbox{$s>\left(\max\{d/2,d/p\}-1\right)$} and let $\F\subset B^{s}_{pq}(\Td)$ be bounded. For $f\in\F$, we write $\overline{f} = f - \int f\,\ud\mu$. Let $\G_T$ be as in (\ref{Eq: empirical process definition}).
	
	\subsection{Convergence of Finite-Dimensional Marginals}
	
	Our first task is to establish the covariance structure of the candidate limit process $\G$. By linearity of the process $f\mapsto\G_T(f)$ and the Cram\'{e}r-Wold device, it suffices to consider a single $f\in\F\cap L^2_{\mu}$. The following argument was first used in \cite{bhattacharyaFunctionalCentralLimit1982}; we require a slightly augmented version for the function classes we consider.
	
	As $f\in B^{s}_{pq}$, by Lemma \ref{Lemma: 2-smoothing estimate} we have that $L^{-1}[f]\in B^{2+s}_{pq}$. Using standard Besov embeddings (see, for example, \cite[Proposition 4.3.10]{gineMathematicalFoundationsInfinitedimensional2016}) and the Sobolev-Gagliardo-Nirenberg inequality, for $p\geq2$ we have the sequence of continuous embeddings
	$ B^{s+2}_{pq}\subset B^{s+2-\delta}_{22} = W^{s+2-\delta,2} \subset W^{2,r} $
	for some small $\delta>0$ and some $r>d$. If $1\leq p<2$, there is an extra initial step: we use the embedding 
	$B^{s+2}_{pq}\subset B^{s-\frac{d}{p}+\frac{d}{2}+2}_{2q}$
	and then the fact that $s>\frac{d}{p}-1$.
	Thus for $f\in\F$, $L^{-1}[f]\in W^{2,r}(\Td)$ and so we may apply the It\^{o}-Krylov formula \cite[Theorem 2.10.1]{krylovControlledDiffusionProcesses1980} to obtain
	\begin{equation}\label{Eq: Ito decomposition of G_T}
		\G_T(f) = \frac{1}{\sqrt{T}}\left(L^{-1}[f](X_T) - L^{-1}[f](x)\right) - \frac{1}{\sqrt{T}}\int_0^T\nabla L^{-1}[f](X_s)^\top\sigma(X_s)\ud W_s. 
	\end{equation}
	Observe that $\sup_{f\in\F}\|L^{-1}[f]\|_{\infty}<\infty$ so the first term is $O_P(T^{-1/2})$, and is thus asymptotically negligible.
	Now let $T_n$ be any sequence increasing to infinity. Define processes
	$$M^n_t := \frac{1}{\sqrt{T_n}}\int_0^{T_nt}\nabla L^{-1}[f](X_s)^{\top}\sigma(X_s)\cdot \ud W_s, \quad t\geq0. $$
	As $L^{-1}[f]\in C^1$ by a further Sobolev embedding, each $M^n$ is in fact a continuous martingale with quadratic variation given by
	\begin{equation*}
		\langle M^n \rangle_t = \frac{1}{T_n}\int_0^{T_nt}\left\| \sigma(X_s)^\top\nabla L^{-1}[f](X_s)\right\|_2^2\,\ud s \,\, \overset{P}{\longrightarrow} \,\, t\|\sigma^\top \nabla L^{-1}[f]\|^2_{\mu}
	\end{equation*}
	as $n\to\infty$, using the ergodic average property (\ref{Eq: ergodic property of mu, functions}), which is proved for $f\in\F$ as in \cite[Lemma 6]{nicklNonparametricStatisticalInference2020}. Then by the martingale central limit theorem (\cite[Theorem 7.1.4]{ethierMarkovProcessesCharacterization1986}), as $T\to\infty$ we have that
	$$ \frac{1}{\sqrt{T}}\int_0^T\nabla L^{-1}[f](X_s)^\top\sigma(X_s)\cdot\ud W_s \to^{\law} N\left(0, \|\sigma^\top \nabla L^{-1}[f]\|^2_{\mu}\right),$$
	and so by (\ref{Eq: Ito decomposition of G_T}) and linearity, we have for any $k\in\mathbb{N},f_1,\ldots,f_k\in\F$, that as $T\to\infty$,
	$\left(\G_T(f_1),\ldots,\G_T(f_m)\right) \to^{\law} N_m(0,C),$
	where the covariance matrix $C$ is given by
	$ C_{ij} = \left\langle \sigma^\top\nabla L^{-1}[\bar{f_i}],\sigma^{\top}\nabla L^{-1}[\bar{f_j}] \right\rangle_{\mu}$.
	Thus the finite-dimensional marginals of $\G_T$ converge to those of $\G$ as defined by \eqref{Eq: limit process cov kernel}.
	
	\subsection{Asymptotic Equicontinuity}\label{Subsection: Proof, Asymptotic Equicontinuity}
	
	It now remains to prove that the asymptotic equicontinuity condition (\ref{Eq: asymptotic equicontinuity condition}) holds for $\G_T$, for some pseudo-distance $\rho$ with respect to which $\F$ is totally bounded. By Markov's inequality, it is sufficient to prove that
	\begin{equation}\label{Eq: asymptotic equicontinuity, expectation version}
		\lim_{\delta\to0}\limsup_{T\to\infty}\,E\left[ \sup_{f,g\in\F:\rho(f,g)\leq\delta} \big|\G_T(f)-\G_T(g)\big| \right] = 0.
	\end{equation}
	We first address the choice of pseudo-distance. Recall the decomposition of $\G_T(f)$ in (\ref{Eq: Ito decomposition of G_T}):
	\begin{equation*}
		\G_T(f) = \frac{1}{\sqrt{T}}\underbrace{\left(L^{-1}[f](X_T) - L^{-1}[f](x)\right) }_{=:Y_T(f)}- \frac{1}{\sqrt{T}}\underbrace{\int_0^T\nabla L^{-1}[f](X_s)^\top\sigma(X_s)\ud W_s}_{=:Z_T(f)}
	\end{equation*}
	where for fixed $f\in\F$, $(Z_t(f))_{t\geq0}$ is a continuous local martingale. As noted previously, $T^{-1/2}Y_T(f)$ is asymptotically negligible and it suffices to prove the asymptotic equicontinuity condition (\ref{Eq: asymptotic equicontinuity, expectation version}) with $\frac{1}{\sqrt{T}}Z_T$ in place of $\G_T$. 
	
	The quadratic variation of $Z(f)$ can be bounded above:
	$$ \langle Z(f) \rangle_T = \int_0^T \|\sigma(X_s)^{\top}\nabla L^{-1}[f](X_s)\|^2\,\ud s \leq \Lambda \int_0^T \|\nabla L^{-1}[f](X_s)\|^2\,\ud s,$$
	where $\Lambda<\infty$ is as in (\ref{Eq: strict ellipticity definition}). Thus from the argument of Lemma 1 in \cite{nicklNonparametricStatisticalInference2020}, the process $\left(T^{-1/2}Z_T(f)\right)_{f\in\F}$ is sub-Gaussian with respect to the pseudo-distance $\rho_L$ given by
	\begin{equation}\label{Eq: d_L definition}
		\rho_L^2(f,g) = \Lambda \sum_{i=1}^d \left\| \partial_i L^{-1}[f-g]\right\|_{\infty}^2.
	\end{equation}
	The next lemma, similar to Lemma 3 of \cite{nicklNonparametricStatisticalInference2020}, establishes a useful upper bound for $\rho_L$.
	
	\begin{lemma}\label{Lemma: Besov bound for d_L}
		Let $b,\sigma\in C^{1+\beta}(\Td)$ for some $\beta>0$. Then for any non-integer $\gamma$ such that $0<\gamma\leq\beta$,
		$$ \rho_L(f,g) \lesssim \|f-g\|_{B^{-1+\gamma}_{\infty\infty}},$$
		where the constant depends on $b,\sigma,d,\beta,\gamma$.
	\end{lemma}
	Thus $T^{-1/2}Z_T$ is also sub-Gaussian with respect to (a multiple of) the $B^{-1+\gamma}_{\infty\infty}$-norm, for arbitrarily small $\gamma>0$; we now confirm (\ref{Eq: asymptotic equicontinuity, expectation version}) with $T^{-1/2}Z_T$ in place of $\G_T$, with respect to the $B^{-1+\gamma}_{\infty\infty}$-norm. Note that $\F$ is totally bounded with respect to this norm so long as $\gamma \leq s+1$, which we assume. For any pseudo-distance $\rho$, let $N(\F,\rho,\varepsilon)$ denote the covering number of $\F$ by $\rho$-balls of radius $\varepsilon$. By Dudley's chaining bound for sub-Gaussian processes (see e.g. \cite[Theorem 2.3.7]{gineMathematicalFoundationsInfinitedimensional2016}), if
	\begin{equation}\label{Eq: entropy integral}
		\int_0^{\infty} \sqrt{\log{N(\F,\|\cdot\|_{B^{-1+\gamma}_{\infty\infty}},\varepsilon)}}\,\ud\varepsilon <\infty,
	\end{equation}
	then
	\begin{align*}
		E\left[ \sup_{f,g\in\F:\|f-g\|_{B^{-1+\gamma}_{\infty\infty}}\leq\delta}\frac{1}{\sqrt{T}} \big|Z_T(f)-Z_T(g)\big| \right] &\lesssim \int_0^{\delta}  \sqrt{\log{2N(\F,\|\cdot\|_{B^{-1+\gamma}_{\infty\infty}},\varepsilon)}}\,\ud\varepsilon,
	\end{align*}
	where the constant is universal.
	In view of the finiteness of the entropy integral in (\ref{Eq: entropy integral}), the right-hand side converges to zero as $\delta\to0$. Since the constants in the previous display do not depend on $T$, it follows that (\ref{Eq: entropy integral}) implies (\ref{Eq: asymptotic equicontinuity, expectation version}) for $T^{-1/2}Z_T$, and thus for $\G_T$. To prove \eqref{Eq: entropy integral}, we make use of the following covering number bounds.
	
	\begin{lemma}\label{Lemma: covering number bounds}
		Let $1\leq p,p'\leq\infty$ and let $s,t\in\R$ be such that $s-t>\max\left(\frac{d}{p}-\frac{d}{p'},0\right)$. Let $\mathcal{B}=\mathcal{B}(s,p,M)$ be the norm-ball of radius $M$ in $B^s_{p\infty}(\Td)$. Then $\mathcal{B}$ satisfies
		$$ \log N(\mathcal{B}, \|\cdot\|_{B^t_{p'1}},\varepsilon) \simeq \left(\frac{M}{\varepsilon}\right)^{\frac{d}{s-t}} \quad \forall \varepsilon>0,$$
		where the constants depend on $p,p',s,t,d$.
		\begin{remark}\label{key}
			Using the embeddings $B^s_{pq}\subset B^s_{p\infty}$ and $B^t_{p'1}\subset B^t_{p'q'}$ for any $1\leq q,q'\leq \infty$, the lemma also holds with any choices of third Besov index.
		\end{remark}
	\end{lemma}
	
	We now bound the entropy integral: letting $D$ be the \mbox{$\|\cdot\|_{B^{-1+\gamma}_{\infty\infty}}$}-diameter of $\F$ (which is certainly finite), for $\varepsilon>D$ the integrand in (\ref{Eq: entropy integral}) is zero and so the range of integration is effectively $[0,D]$. Since $\F\subset B^{s}_{pq}$ is bounded, we may then apply Lemma \ref{Lemma: covering number bounds} with $t=-1+\gamma,p'=\infty$ (here we require the assumption $s>d/p-1$) to obtain that
	\begin{align*}
		\int_0^D \sqrt{\log{N(\F,\|\cdot\|_{B^{-1+\gamma}_{\infty\infty}},\varepsilon)}}\,\ud\varepsilon &\lesssim \int_0^D \varepsilon^{-\frac{d}{2(1+s-\gamma)}}\,\ud\varepsilon;
	\end{align*}
	this integral is finite if and only if $\frac{d}{2(1+s-\gamma)}<1$, or equivalently $d<2(s+1-\gamma)$. Since we may choose $\gamma>0$ arbitrarily small, the entropy integral (\ref{Eq: entropy integral}) can be made finite whenever $d<2(s+1)$, or equivalently (recalling that $d\geq2$), whenever $s>\frac{d}{2}-1$.
	
	\section{Additional Proofs}\label{Section: additional proofs}
	
	\begin{proof}[Proof of Proposition \ref{Prop: sharpness of CLT}]
		Without loss of generality, we may assume that $\F = \F\cap L^2_{\mu}$.
		
		We first consider the case $s>\frac{d}{2}-1$. Recall the \emph{intrinsic pseudo-distance} $\rho_{\G}$ induced by the Gaussian process $\G$, defined by $\rho_{\G}^2(f,g) = E(\G(f) - \G(g))^2$. Then $\G$ is sub-Gaussian with respect to $\rho_{\G}$ (see, for example, the discussion after Definition 2.3.5 in \cite{gineMathematicalFoundationsInfinitedimensional2016}). By using the equivalence of $\mu$ to Lebesgue density, the strict ellipticity of $\sigma$, and Lemma \ref{Lemma: 2-smoothing estimate}, we have that for $f,g\in L^2_{\mu}$,
		$\rho_{\G}(f,g) \leq C(b,\sigma,d)\|f-g\|_{H^{-1}}^2.$
		Thus $\G$ is sub-Gaussian with respect to the $H^{-1}$-norm; to prove tightness of $\G$ as a measure on $\ell^{\infty}(\F)$, it suffices to check condition \eqref{Eq: entropy integral} with $\|\cdot\|_{H^{-1}}$ in place of $\|\cdot\|_{B^{-1+\gamma}_{\infty\infty}}$. For any values of $p,q$ we may apply the covering number bound Lemma \ref{Lemma: covering number bounds} to obtain that
		\[ \int_0^{\infty} \sqrt{\log N(\F,\|\cdot\|_{H^{-1}},\varepsilon)}\,\ud\varepsilon \lesssim \int_0^D \varepsilon^{-\frac{d}{2(1+s)}} \ud\varepsilon, \]
		where $D>0$ is the $H^{-1}$-diameter of $\F$. This integral is finite since $s>\frac{d}{2}-1$, confirming condition \eqref{Eq: entropy integral}. Thus $\G$ is a tight Borel measure on $\ell^{\infty}(\F)$.
		
		Next, we consider the case $s<\frac{d}{2}-1$; then for some $\delta>0$, we have $s = \frac{d}{2}-1-\delta$. Clearly it suffices to consider the case where $p=\infty$ and $\delta>0$ arbitrarily small. By Sudakov's Theorem \cite[Corollary 2.4.13]{gineMathematicalFoundationsInfinitedimensional2016}, if
		\begin{equation}\label{Eq: Sudakov thm condition}
			\liminf_{\varepsilon\to0}\varepsilon\sqrt{\log N(\F,\rho_{\G},\varepsilon)} = \infty,
		\end{equation}
		then $\sup_{f\in\F}|\G(f)| = \infty$ a.s., and so $\G$ is not sample bounded; in particular, it is not tight.
		
		Using the strict ellipticity assumption on $\sigma\sigma^{\top}$ \eqref{Eq: strict ellipticity definition}, we have that for $f,g\in\F$,
		\begin{equation}\label{Eq: d_G LB, midpoint}
			\rho_{\G}^2(f,g) =  \left\| \sigma^{\top}\nabla L^{-1}[f-g] \right\|_{\mu}^2 \gtrsim \left\| \nabla L^{-1}[f-g] \right\|_{\mu}^2.
		\end{equation}
		Now, $L^{-1}[f-g]\in L^2_{\mu}$, and so the Poincar\'{e} inequality (\cite[Chapter 7]{gilbargEllipticPartialDifferential2001}) with reference density $\mu$ yields that
		$\|L^{-1}[f-g]\|_{\mu}\lesssim \| \nabla L^{-1}[f-g]\|_{\mu}.$
		Hence, again using the equivalence of $\|\cdot\|_{L^2}$ and $\|\cdot\|_{\mu}$,
		$$ \|L^{-1}[f-g]\|_{H^1}^2 = \|L^{-1}[f-g]\|_{L^2}^2 + \|\nabla L^{-1}[f-g]\|_{L^2}^2 \lesssim \|\nabla L^{-1}[f-g]\|_{L^2}^2. $$
		In combination with (\ref{Eq: d_G LB, midpoint}) and the fact that $L:H^1\to H^{-1}$ is a continuous linear operator (see \cite{bersPartialDifferentialEquations1964}, Part II, Lemma 3.9), we deduce that
		$\rho_{\G}(f,g) \gtrsim \|f-g\|_{H^{-1}}.$
		It therefore suffices to prove (\ref{Eq: Sudakov thm condition}) with $\|\cdot\|_{H^{-1}}$ in place of $\rho_{\G}$. 
		Recall that $\F$ is the unit ball of $B^s_{pq}$ (intersected with $L^2_{\mu}$); we may assume that $s>0$ and then apply Lemma \ref{Lemma: covering number bounds} to obtain that for $\varepsilon>0$ sufficiently small,
		$ \log N(\F,\|\cdot\|_{H^{-1}},\varepsilon) \gtrsim \varepsilon^{-\frac{d}{s+1}}, $
		and so using that $s = d/2-1-\delta$ for some $\delta>0$,
		$$  \varepsilon\sqrt{\log N(\F,\|\cdot\|_{H^{-1}}, \varepsilon)}  \gtrsim \varepsilon^{1-\frac{d}{2(s+1)}}  = \varepsilon^{-\frac{2\delta}{d - 2\delta}}.$$
		This quantity diverges as $\varepsilon\to0$, establishing (\ref{Eq: Sudakov thm condition}) when $s<\frac{d}{2}-1$.
		
		The case $s=\frac{d}{2}-1$ follows a slightly more refined version of the argument for $s<\frac{d}{2}-1$, as outlined in Remark 12 of \cite{nicklBernsteinMisesTheorems2020}; we do not give details here.
	\end{proof}
	
	\begin{proof}[Proof of Lemma \ref{Lemma: Besov bound for d_L}]
		For any $\gamma>0$ such that $\gamma\not\in\mathbb{N}$ and any function $h$, standard Besov embeddings give that
		$ \sum_{i=1}^d \|\partial_i g\|_{\infty}^2 \leq \|h\|_{C^1}^2 \leq \|h\|_{C^{1+\gamma}}^2 \lesssim \|h\|^2_{B^{1+\gamma}_{\infty\infty}},$
		for a constant depending on $d,\gamma$ only. Thus
		$ \rho_L(f,g) \lesssim \|L^{-1}[f-g]\|^2_{B^{1+\gamma}_{\infty\infty}}.$
		The result then follows from an application of Lemma \ref{Lemma: 2-smoothing estimate}.
	\end{proof}
	
	\begin{proof}[Proof of Lemma \ref{Lemma: covering number bounds}]
		By a scaling argument, we need only consider the case $M=1$. Recall the \emph{entropy numbers} $e_k,k\geq1$ of $\mathcal{B}$: $e_k$ is defined to be the smallest radius $r$ such that $\mathcal{B}$ can be covered by at most $2^k$ $\rho$-balls of radius $r$.
		
		Theorem 2 in \cite[Section 3.3]{edmundsFunctionSpacesEntropy1996} establishes that for $\mathcal{B}$ as in the statement of the lemma, the entropy numbers $e_k$ satisfy
		$e_k \simeq k^{-\frac{s-t}{d}},$
		with constants as stated. The result now follows from the fundamental relationship between entropy numbers and covering numbers (see Exercise 2.5.5 and the preceding discussion in \cite{talagrandUpperLowerBounds2014}):
		$$ e_k \simeq k^{\alpha} \Leftrightarrow \log N(\mathcal{B},\rho,\varepsilon) \simeq \varepsilon^{1/\alpha}.$$
	\end{proof}
	
	\appendix
	
	\section{A PDE Regularity Estimate}
	
	For a more detailed exposition of the PDE theory of the generator $L$ as defined in (\ref{Eq: generator definition}) and its adjoint $L^*$, we refer to \cite[Section 6]{nicklNonparametricStatisticalInference2020} and \cite[Chapter II.3]{bersPartialDifferentialEquations1964}.
	
	Consider the Poisson equation
	\begin{equation}\label{Eq: Poisson equation}
		Lu = f \quad \text{in }\Td
	\end{equation}
	where $f\in L^2(\Td)$ is given. A key property of the solution operator $L^{-1}$ defined by (\ref{Eq: solution operator definition}) is that it is `2-smoothing': see, for example, \cite[Chapter 3, Theorem 3]{bersPartialDifferentialEquations1964}. We require a quantitative version of this property, similar to equation (74) in \cite{nicklNonparametricStatisticalInference2020} or Lemma 8 in \cite{giordanoNonparametricBayesianInference2022}. We reformulate these results slightly for more general Besov spaces.
	
	\begin{lemma}\label{Lemma: 2-smoothing estimate}
		Let $t\geq2$. Assume that $b\in C^{t-2}(\Td)$, and that $\sigma\in C^{t-2}(\Td)$ is such that the strict ellipticity condition (\ref{Eq: strict ellipticity definition}) is satisfied with constant $\lambda>0$. Then for any $f\in L^2_{\mu}(\Td)$, there exists a unique solution $L^{-1}[f]\in L^2_0(\Td)$ of (\ref{Eq: Poisson equation}) such that $LL^{-1}[f] = f$ almost everywhere. Moreover, for any $s\leq t$ and any $1\leq p,q \leq \infty$,
		$$ \left\| L^{-1}[f]\right\|_{B^s_{pq}} \lesssim \|f\|_{B^{s-2}_{pq}}, $$
		where the constant depends on $s,t,p,q,d,\lambda$ and an upper bound for $\|b\|_{B^t_{\infty\infty}}$.
	\end{lemma}
	The lemma is proved exactly as in Section 6.0.1 of \cite{nicklNonparametricStatisticalInference2020}, with the generalisation of the Besov indices requiring only cosmetic changes, once one obtains the estimate
	\begin{equation}\label{Eq: 2-smoothing key equation}
		\|u\|_{B^s_{pq}} \lesssim \|Au\|_{B^{s-2}_{pq}} \quad \forall u\in B^s_{pq}\cap L^2_0,
	\end{equation}
	where $A$ is the second-order differential operator given by
	$A = \sum_{i,j=1}^d a_{ij}(\cdot)\partial_i\partial_j$.
	This replaces an analogous estimate in \cite{nicklNonparametricStatisticalInference2020} for the case $A=\Delta$, the Laplacian. One proves the estimate (\ref{Eq: 2-smoothing key equation}) as in that reference, substituting the Fourier coefficient identity
	$ \langle \Delta u, e_k \rangle = -(2\pi)^2|k|^2\langle u,e_k \rangle$
	with the inequality
	$ |\langle Au, e_k \rangle| \geq (2\pi)^2\lambda |k|^2 |\langle u,e_k\rangle|$, which holds for $k\in\Zd\setminus\{0\}$.
	Plugging this into a Littlewood-Paley definition of the $B^s_{pq}$-norm and using a Fourier multiplier argument then yields (\ref{Eq: 2-smoothing key equation}), and thus Lemma \ref{Lemma: 2-smoothing estimate}.

	\section*{Acknowledgements}
		The author gratefully thanks Richard Nickl for his guidance and advice in this project, and also Randolf Altmeyer for helpful discussions.

	\newpage
	\bibliographystyle{plain}
	\bibliography{DiffusionCLTRefs}

\end{document}